\documentclass[11pt]{amsart}
\usepackage{amsmath,amssymb,amsfonts,amsthm}  
\usepackage{graphics}  
\usepackage[all,cmtip]{xy}
\usepackage{hyperref}

\usepackage{xcolor}
\definecolor{darkgreen}{RGB}{0, 100, 0}

\usepackage{tikz}
\usetikzlibrary{arrows.meta, positioning}

\newcommand{\co}{\colon\,}

\usepackage{color}

\newtheorem{thm}{Theorem}[section]
\newtheorem*{thm*}{Theorem}

\newtheorem{cor}[thm]{Corollary}
\newtheorem{lemma}[thm]{Lemma}
\newtheorem{rem}[thm]{Remark}
\newtheorem{prop}[thm]{Proposition}

\theoremstyle{definition}

\newcommand{\ds}{\displaystyle}

\usetikzlibrary{cd}
\usetikzlibrary{arrows,automata,shapes,patterns,calc}
\usepackage{tkz-graph}
\usetikzlibrary{decorations.markings}
\usetikzlibrary{decorations.pathreplacing}
\usetikzlibrary{arrows.meta}
\usetikzlibrary{bending}

\begin{document}
﻿\title{A Classification of Bicritical Dynamic Portraits}
\author[Saenz]{Edgar Saenz}
\address{Department of Mathematics\\ Virginia Tech\\
Blacksburg, VA 24061\\ U.S.A.}
\email{easaenzm@math.vt.edu}
﻿
\author[Samji]{Dheemanth Samji}
\address{Department of Mathematics\\ Virginia Tech\\
Blacksburg, VA 24061\\ U.S.A.}
\email{dheemanth@vt.edu}
﻿
\date\today
\maketitle
﻿
\begin{abstract} 
The dynamic of a rational 
map $f:\widehat{\mathbb C}\to \widehat{\mathbb C}$ is determined by the forward orbits of its critical 
points. Such a map is called {\em postcritically finite} 
if every critical point has finite forward orbit, or equivalently, 
if every critical point eventually maps into a periodic cycle. 
These orbits can be encoded in a finite 
directed graph called a {\em dynamic portrait}. In this work, 
we classify which abstract bicritical dynamic portraits of degree $d\geq2$ with at least 4 postcritical points are realizable exclusively by rational Thurston maps.
\end{abstract}

\section{Introduction}

\noindent 
Let $\widehat{\mathbb C}$ denote the Riemann sphere, and let 
$f:\widehat{\mathbb C}\to \widehat{\mathbb C}$ be a rational map of 
degree $d\geq 2$. By the Riemann-Hurwitz formula, such a map has $2d-2$ 
critical points, counted with multiplicity; we denote this set by $C_f$. The {\em postcritical set} $P_f$, is defined as the union of the forward orbits of the critical points. If $P_{f}$ is finite, then $f$ is called {\em postcritically finite}. 

To any postcritically finite rational map $f$, one can associate a finite directed graph $\Gamma_{f}$, called its {\em dynamic portrait}, which encodes the action of 
$f$ restricted to $C_f\cup P_f$. As an illustration, consider the polynomial 
$f:z\mapsto  z^2-i$. Its critical set is $C_f=\{0,\infty\}$, and its  
postcritical set is $P_f=\{-i,-1-i,i,\infty\}$. The corresponding dynamic portrait is obtained by placing a directed edge from a vertex $x$ to $f(x)$, labeled by the local degree $\deg_{f}(x)$ whenever this degree exceeds one.
\[\xymatrix{0\ar[r]^{2}&-i\ar[r]&-1-i\ar[r]&i\ar@/^1pc/[l]}\hspace{.45in}\xymatrix{\infty\ar@(ur,dr)^{2}}\]

\noindent In this example, the portrait is both a {\em polynomial portrait}, since $\infty$ is a fixed point mapping to itself with full degree, and a {\em bicritical portrait}, since two vertices (namely $0$ and $\infty$)
map with full degree. We are interested in determining which finite weighted directed graphs arise as dynamic portraits of postcritically finite bicritical rational maps.

A weighted finite directed graph with exactly two critical points that satisfies the Riemann-Hurwitz restriction (see Section \ref{sect:preliminaries}) is called an \emph{abstract bicritical portrait}. The central goal of this work is to understand which such abstract portraits can be realized by Thurston maps and, more specifically, to identify those that are realized only by rational maps.

\noindent{\bf Thurston's Theorem.} Let $S^2$ denote an oriented topological 
2-sphere, and let $f\co S^2 \to S^2$ be 
an orientation-preserving branched cover of degree $d\ge 2$ so that 
the postcritical set $P_f$ is finite. We call such a map $f$ a 
{\em Thurston map}.  Two Thurston maps  $f:(S^2,P_f)\to (S^2,P_f)$ 
and $g:(S^2,P_g)\to (S^2,P_g)$ are {\em combinatorially equivalent} 
if there are orientation-preserving homeomorphisms 
$h_0, h_{1}:(S^2,P_f)\to (S^2,P_g)$ such that $h_0$ is isotopic to $h_{1}$ relative to $P_f$ and 
 $h_0\circ f = g\circ h_1$. The {\em orbifold $\mathcal{O}_{f}$ associated} to $f$ is the topological orbifold which has underlying space $S^2$ and whose weight $\nu(x)$ at $x\in S^2$ is given by the least common multiple of the local degree of $f$ over all the iterated preimages of $x$. The orbifold $\mathcal{O}_{f}$ is said to be hyperbolic if the Euler characteristic $\chi(\mathcal{O}_{f})=2-\sum_{x\in P_{f}}(1-1/\nu(x))$ is negative. 

In the early 1980's, William Thurston addressed the problem of determining when a Thurston map $f$ is combinatorially equivalent to a rational map $F:\widehat{\mathbb{C}}\to\widehat{\mathbb{C}}$. To state his topological characterization, when $\mathcal{O}_{f}$ is hyperbolic, we first need some terminology.

A {\em multicurve} $\Delta$ is a finite collection of simple disjoint curves in $S^2\setminus P_f$, no two of which are homotopic in $S^{2}\setminus P_{f}$. All components $\delta\in \Delta$ are also required to be {\em essential}; that is, each connected component of $S^{2}\setminus \delta$  contains at least two points of $P_{f}$. The multicurve $\Delta$ is 
said to be $f$-$stable$ if for all 
$\delta\in\Delta$, we have that every essential component of $f^{-1}(\delta)$ is homotopic in $S^2\setminus P_f$ to some $\delta'\in\Delta$. 

Given an $f$-stable multicurve $\Delta$, Thurston defined a matrix $A^{\Delta}$ that encodes how different components of $f^{-1}(\Delta)$ map to $\Delta$. The matrix $A^{\Delta}:\mathbb{R}^{|\Delta|}\to\mathbb{R}^{|\Delta|}$ is defined in coordinates by 
$$A^{\Delta}_{\gamma\delta}=\ds\sum_{\alpha}\ds\frac{1}{\deg(f:\alpha\to\delta)},$$
where the sum is taken over components $\alpha$ of $f^{-1}(\delta)$ which are homotopic to $\gamma$ in $S^{2}\setminus P_{f}$. The matrix $A^{\Delta}$ has non-negative real entries, by the Perron-Frobenius theorem its spectral radius $\lambda_{\Delta}$ is a positive real eigenvalue with non-negative eigenvector. The multicurve $\Delta$ is an {\em obstruction} provided that $\lambda_{\Delta}\geq 1$. 
If the Thurston map $f$ admits an obstruction, $f$ is said to be 
{\em obstructed}. If not, $f$ is said to be {\em unobstructed}. 

\begin{thm*}[Thurston's Topological Characterization \cite{DH}] Let $f:(S^2,P_{f})\to (S^2,P_{f})$ be a Thurston map with hyperbolic orbifold. Then $f$ is combinatorially equivalent to a rational map
$F$ if and only if $f$ is unobstructed. In this case, the rational map $F$ is unique up to conjugation by M\"obius transformations.
\end{thm*}

\noindent {\bf Levy cycles.} Checking for the existence of Thurston obstructions is highly challenging, as one must generally search through infinitely many multicurves. However, for the family of bicritical Thurston maps, the criterion for a Thurston map to be combinatorially equivalent to a rational map simplifies significantly (see \cite{TanL}). Specifically, searching for arbitrary Thurston obstructions reduces to searching for Levy cycles.

A {\em Levy cycle} for the Thurston map $f:(S^2,P_f)\to (S^2,P_f)$ is 
a circularly ordered multicurve $\Delta=\{\delta_0,\ldots, \delta_{n-1},\delta_{n}=\delta_0\}$ on $S^2\setminus P_f$ such that for every  $i$ at least one component $\delta'_{i-1}$ of $f^{-1}(\delta_i)$ is homotopic to $\delta_{i-1}$ in $S^{2}\setminus P_{f}$, and $\deg(f:\delta'_{i-1}\to \delta_{i})=1$.

\begin{thm*}[Lei \cite{TanL}] Let $f:(S^2,P_f)\to (S^2,P_f)$ be a bicritical Thurston map. Then $f$ is combinatorially equivalent to a rational map if and only if $f$ does not have a Levy cycle. 
﻿\end{thm*}

Motivated by this result we have produced the list of all the abstract bicritical dynamic portraits that are only realizable by rational maps.

\begin{thm}\label{main:thm:1} Let $d\geq 2$ and let $f$ be a Thurston map with $|P_{f}|\geq 4$ and whose dynamic portrait is isomorphic to any of the following abstract bicritical dynamic portraits. 

\begin{itemize}
\item[(1)] $\xymatrix{x_{1}\ar[r]^{d}&x_{2}\ar[r]&\cdots\ar[r]&x_{n}\ar@/^1pc/[lll]}$\hspace{.45in}$\xymatrix{\infty\ar@(ur,dr)^{d}}$\vspace{.1in}

\item[(2)] $\xymatrix{\star\ar[r]^{d}&x_{1}\ar[r]&\cdots\ar[r]&x_{n}\ar@(ur,dr)}$\hspace{.5in}$\xymatrix{\infty\ar@(ur,dr)^{d}}$\vspace{.1in}

\item[(3)] $\xymatrix{x_{n}\ar[r]^{d}&a\ar[r]^{d}&b\ar[r]&x_{1}\ar[r]&\cdots\ar[r]&x_{n-1}\ar@/^1pc/[lllll]}$\vspace{.1in}

\item[(4)] $\xymatrix{\star\ar[r]^{d}&a\ar[r]^{d}&b\ar[r]&x_{1}\ar[r]&\cdots\ar[r]&x_{n}\ar@(ur,dr)}$\vspace{.1in}

\item[(5)] $\xymatrix{\star\ar[r]^{d}&a\ar[r]&x_{1}\ar[r]&x_{2}\ar[r]&\cdots\ar[r]&x_{n}\ar[r]^{d}&b\ar@/^1pc/[llll]}$\vspace{.1in}

\item[(6)] $\xymatrix@R=0.05pc{
\star \ar[r]^d & a \ar[dr] & & & & \\
& & x_1 \ar[r] & x_2 \ar[r] & \cdots \ar[r] & x_n \ar@(ur,dr) [] \\
\star \ar[r]^d & b \ar[ur] & & & &
}$

\item[(7)] $\xymatrix{\star\ar[r]^{2}&x\ar[r]&y\ar[r]&x_{1}\ar[r]^{2}&x_{2}\ar[r]&x_{3}\ar@/^1pc/[ll]}$
\end{itemize}

Then $f$ is combinatorially equivalent to a rational map. 
\end{thm}
\noindent To prove Theorem \ref{main:thm:1}, one needs to analyze each bicritical portrait to show that such a map $f$ cannot have Levy cycles. It then follows from Lei's Theorem that $f$ is {\em combinatorially equivalent} to a rational map.

\begin{thm}\label{main:thm:2} Suppose that $\Gamma$ is an abstract bicritical dynamic portrait of degree $d$, with at least four postcritical vertices and that satisfies one of the following properties:
\begin{itemize}
\item[(i)] $d\geq2$ and $\Gamma$ contains a nonattracting cycle of length at least 2; 
\item[(ii)] $d\ge3$ and $\Gamma$ contains a subgraph isomorphic to 
\begin{eqnarray*}
\xymatrix@!R=3pt@!C=7pt{
\star\ar[r]^{d} &a\ar[r]&p\\}\qquad 
\xymatrix@C=13pt{\star\ar[r]^{d} &b\ar[r]&q\\}
\end{eqnarray*}
with $p\neq q$ (this includes the case $\xymatrix{\star\ar[r]^{d}&a\ar[r]&p\ar[r]^{d}&b\ar[r]&q}$).
 \end{itemize}
 Then $\Gamma$ is realized by a Thurston map that has a Levy cycle.
\end{thm}

\noindent To prove Theorem \ref{main:thm:2}, we construct a 
bicritical Thurston map $f:S^2 \to S^2$ so that it has a Levy cycle and $\Gamma_f\simeq \Gamma$. By Lei's Theorem, it follows that $f$ is not combinatorially equivalent to a rational map.

Using formal matings of quadratic polynomials and Lemma \ref{lem:reductionlem}, in Section \ref{sect:deg2} we determine the bicritical dynamic portraits of degree 2 that can be realized by an obstructed Thurston map. Combining Theorems \ref{main:thm:1} and \ref{main:thm:2}, in Section  \ref{sect:cubp} we classify the completely unobstructed abstract bicritical dynamic portraits of degree $d\geq2$.

\section{Preliminaries}\label{sect:preliminaries}
\noindent {\bf Portraits associated to Thurston maps.} Let
$f:(S^2,P_f)\to (S^2,P_f)$ be a Thurston map of degree $d$. The
\emph{dynamic portrait} of $f$ is the weighted directed graph
$\Gamma$ such that the vertex set $V(\Gamma)$ is the union of the set
$C_f$ of critical points and the set $P_f$ of postcritical points, and
for each vertex $v$ there is an edge from $v$ to $f(v)$ with weight
the local degree $\deg_f(v)$ of $f$ at $v$.  By the Riemann-Hurwitz
formula,
\[
\sum_{v\in C_f} (\deg_f(v)-1) = 2d-2.
\]
Since $f$ has degree $d$, at each vertex $v$ the sum of the weights of
the incoming edges is at most $d$. We say that $f$ is a topological
polynomial if and only if there is a vertex $v$ such that $f(v) = v$
and $\deg_f(v)=d$.

\noindent {\bf Formal Mating.} Let $\mathbb{S}$ be the unit sphere in $\mathbb{C}\times\mathbb{R}\approx \mathbb{R}^3$. Here we review some basic terminology and results about matings of polynomials following \cite{BEK} and \cite{TanL}. Let $P:\mathbb{C}\to \mathbb{C}$ and $Q:\mathbb{C}\to\mathbb{C}$ be two monic polynomials of the same degree $d\geq 2$. The {\em formal mating} of $P$ and $Q$ is the branched covering $f=P\uplus Q:\mathbb{S}\to\mathbb{S}$ as follows. We identify the dynamical plane of $P$ to the upper hemisphere $\mathbb{H}^{+}$ of $\mathbb{S}$ and the dynamical plane of $Q$ to the lower hemisphere $\mathbb{H}^{-}$ of $\mathbb{S}$ via the gnomonic projections: 
$$\nu_{P}:\mathbb{C}\to\mathbb{H}^{+}\hspace{.5in}\text{and}\hspace{.5in}\nu_{Q}:\mathbb{C}\to\mathbb{H}^{-}$$
given by 
$$\nu_{P}(z)=\ds\frac{(z,1)}{||(z,1)||}=\ds\frac{(z,1)}{\sqrt{|z|^2+1}}\hspace{.5in}\text{and}\hspace{.5in}\nu_{Q}(z)=\ds\frac{(\bar{z},1)}{||(z,1)||}=\ds\frac{(\bar{z},1)}{\sqrt{|z|^2+1}}.$$
The map $\nu_{P}\circ P\circ \nu_{P}^{-1}$ defined on the upper hemisphere and $\nu_{Q}\circ Q\circ \nu_{Q}^{-1}$ defined in the upper hemisphere extend continuously to the equator of $\mathbb{S}$ by 
$(e^{2i\pi \theta},0)\mapsto(e^{2i\pi d\theta},0)$. The two maps fit together so as to yield a branched covering map $P\uplus Q:\mathbb{S}\to\mathbb{S}$ which is called the formal mating of $P$ and $Q$. 

If $P$ and $Q$ are postcritically finite polynomials, then the formal mating $P\uplus Q$ is a Thurston map. Moreover, if $P(z)=z^2+c_{P}$ and $Q(z)=z^{2}+c_{Q}$ are both postcritically finite and $c_{P}$ and $c_{Q}$ are in conjugate limbs of the Mandelbrot set, then the formal mating $f:=P\uplus Q$ has a Levy cycle (see Theorem 4.1 of \cite{TanL}). In particular, if $P(z)=z^2+c_{P}$ is a Misiurewick polynomial or a hyperbolic polynomial\footnote{ Under these settings, $P(z)=z^2+c_{P}$ is said to be hyperbolic if $P^{\circ n}(0)=0$ for some $n\in\mathbb{Z}^{+}$.}, and $Q(z)=z^{2}+c_{Q}$ is a Misiurewick polynomial or a hyperbolic polynomial, with $c_{P}$ and $c_{Q}$ in the 1/2-limb of the Mandelbrot set (i.e., in the unique limb which is its own complex conjugate), then the formal mating $P\uplus Q$ is not combinatorially equivalent to a rational map.

\noindent {\bf Abstract dynamic portraits.} Suppose $\Gamma$ is a finite
weighted directed graph (with the weights positive integers) such that
each vertex of $\Gamma$ is the source of exactly one edge.  Let
$\tau\co V(\Gamma)\to V(\Gamma)$ be the function which takes a vertex
$v$ to the target of the edge with source $v$.  We call the weight of
the edge from $v$ to $\tau(v)$ the \emph{degree} of $\tau$ at $v$ and
denote it by $\deg(v)$.  A vertex $v$ is \emph{critical} if $\deg(v) >
1$, and is \emph{postcritical} if there are a critical vertex $w$ and
a positive integer $k$ such that $\tau^{\circ k}(w) = v$. If $v$ is a
critical vertex, then $\tau(v)$ is called a \emph{critical value}.  We
denote the set of critical vertices by $C_{\Gamma}$, and we denote the
set of postcritical vertices by $P_{\Gamma}$. We say that $\Gamma$ is
an \emph{abstract portrait} if it satisfies the following:
\begin{itemize}
\item every vertex of $\Gamma$ is either critical or postcritical, 
\item there is an integer $d\ge 2$ such that 
$\sum_{v\in C_{\Gamma}} (\deg(v)-1) = 2d-2$, and 
\item for each vertex $v$ the sum of
the weights of the edges with target $v$ is at most $d$. 
\end{itemize}
We call $d$ the
\emph{degree} of the abstract dynamic portrait.
We say that an abstract dynamic portrait $\Gamma$ is \emph{realized} by a
Thurston map $f$ if $\Gamma$ is isomorphic to the portrait of $f$ 
(as weighted directed graphs).  An abstract portrait $\Gamma$ is
\emph{realizable} if it is realized by some Thurston map.

We say that an abstract dynamic portrait $\Gamma$ is {\em completely unobstructed} if and only if $\Gamma$ is only realizable by Thurston maps that are combinatorially equivalent to rational maps. 

An abstract portrait of degree $d$ is an \emph{abstract polynomial
 portrait} if there is a vertex $v$ such that $\tau(v)=v$ and
$\deg(v)=d$.  In this case, we choose such a vertex and call it
$\infty$; the other vertices are called \emph{finite}. 

An abstract portrait of degree $d$ is an \emph{abstract bicritical dynamic
portrait} if there are exactly two vertices $v$ and $w$ such that 
$\deg(v)=\deg(w)=d$.

\begin{rem}\label{remark:rem1} The branch data of every abstract bicritical dynamic portrait of degree $d\geq2$ is realizable by the Thurston map $z\mapsto z^{d}$. By Lemma 9.1 of \cite{FPP}, it follows that every abstract bicritical dynamic portrait of degree $d\geq2$ is realizable by a Thurston map.
\end{rem}

\noindent {\bf Properties of bicritical Thurston maps.} From Lemma 5.1 of \cite{BFH}, it is well-known that if $f$ is a topological polynomial and $D\subset S^2\setminus\{\infty\}$ is a topological disk with  $\infty\notin D$, then every connected component of $f^{-1}(D)$ is  a topological disk. Since any bicritical Thurston map can be written as the postcomposition (or precomposition) of a topological polynomial and an orientation preserving homeomorphism, we have the following lemmas.

\begin{lemma}\label{lemma:lem1}
If $f$ is a bicritical Thurston map and $D$ is a topological disk that contains at most one critical value of $f$, then every connected component of $f^{-1}(D)$ is  a topological disk.
\end{lemma}

\begin{lemma}\label{lemma:lem2} Let $f$ be a bicritical Thurston map of degree $d$. If $\gamma$ is an essential simple closed curve in $S^{2}\setminus P_{f}$ that separates the two critical values of $f$, then $\gamma$ cannot be a member of a Levy cycle.
\end{lemma}
\begin{proof} In fact, if $\gamma$ separates the two critical values of $f$, then $f^{-1}(\gamma)$ has exactly one connected component and $f$ maps this connected component to $\gamma$ with degree $d>1$. So, $\gamma$ cannot be a member of a Levy cycle.
\end{proof}

\section{Proof of Theorem \ref{main:thm:1}}
In this section, we provide the key ingredients for the proof of Theorem \ref{main:thm:1}. From the recent classification of completely unobstructed polynomial portraits (see Theorem 1.10 in  \cite{FKKPS}),  portraits (1) and (2) in Theorem \ref{main:thm:1} are completely unobstructed. We show in detail why portrait (4) is completely unobstructed. The proof for each of the remaining portraits is similar.

\noindent\textbf{Analysis of portrait (4)}
$$\xymatrix{\star\ar[r]^{d}&a\ar[r]^{d}&b\ar[r]&x_{1}\ar[r]&\cdots\ar[r]&x_{n}\ar@(ur,dr)}$$

Let $f$ be a Thurston map that realizes dynamic portrait (4) with $n\geq2$. We will show that $f$ does not admit Levy cycles. Note that $|P_{f}|=n+2$.

\begin{prop}\label{prop:proA} Let $Q\subset\{x_{1},x_{2},\dots x_{n}\}$ with $|Q|=2$. If $\gamma$ is a simple closed curve that separates the sets $Q$ and $P_{f}\setminus Q$, then $\gamma$ cannot be a member of a Levy cycle. 
\end{prop}

\begin{proof} We may assume that $Q=\{x_{i}, x_{j}\}$ with $1\leq i<j\leq n$. Let $\gamma$ be a simple closed curve that separates the sets $Q$ and $P_{f}\setminus Q$. We proceed by induction on $i$ to show that $\gamma$ cannot be member of a Levy cycle.

\emph{Base case:} $i=1$ and $1<j\leq n$. Denote by $x_{1}x_{j}$ the core arc with endpoints $x_{1}$ and $x_{j}$ surrounded by $\gamma$.
\begin{itemize} 
\item[$\bullet$] If $j<n$, then the only essential preimage of $x_{1}x_{j}$ under $f$ is an arc of the form $bx_{j-1}$. Hence the only essential preimage of $\gamma$ separates $a$ and $b$. By Lemma \ref{lemma:lem2}, such a preimage cannot be homotopic to a member of a Levy cycle. Thus, $\gamma$ cannot be a member of a Levy cycle.
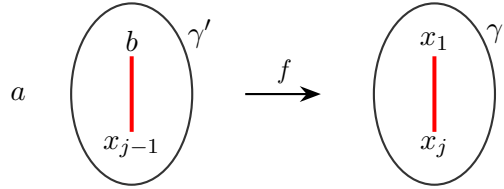
\begin{center}
\begin{tikzpicture}
       \coordinate (P) at (9, 0);
       \coordinate (P1) at (9, -0.5);
       \coordinate (P2) at (9, 0.5);
       \draw[thick, color=black!80] (P) ellipse (0.8cm and 1.2cm); 
       \draw[red, line width=1.5pt] (P1) -- (P2);  
       \node at (7.5, 0) { $a$}; 
       \node at (9, 0.7) { $b$}; 
       \node at (9,-0.7) { $x_{j-1}$}; 
        \node at (9.9, 0.8) { $\gamma'$};
       
       \coordinate (Q) at (13, 0);
       \coordinate (Q1) at (13, -0.5);
       \coordinate (Q2) at (13, 0.5);
       \node at (13, 0.7) { $x_{1}$}; 
       \node at (13, -0.7) { $x_{j}$};
       \node at (13.8, 0.8) { $\gamma$};  
       \draw[thick, color=black!80] (Q) ellipse (0.8cm and 1.2cm);   
    \draw[thick, -{Stealth[scale=1.2]}] (10.5, 0) -- (11.5, 0) 
        node[midway, above, font=\small\itshape] {f};
          \draw[red, line width=1.5pt] (Q1) -- (Q2);
\end{tikzpicture}
\end{center}
\item[$\bullet$] If $j=n$, then the only essential preimage of $x_{1}x_{j}$ under $f$ is an arc of the form $bw$, where either $w=x_{n}$ or $w=x_{n-1}$. In either case, the only essential preimage of $\gamma$ under $f$ separates $a$ and $b$. As in the preceding paragraph, $\gamma$ cannot be a member of a Levy cycle.
\end{itemize}

\emph{Inductive Hypothesis:} Suppose that the statement is valid for any simple closed curve separating the sets $\{x_{i},x_{j}\}$ and $P_{f}\setminus\{x_{i},x_{j}\}$, with $1\leq i<j\leq n$.

Now let $\gamma$ be a simple closed curve separating $u=x_{i+1}$ and $v=x_{j}$ from the other postcritical points, with $i+1<j\leq n$. Let $uv$ be the core arc surrounded by $\gamma$. 
\begin{itemize}
\item[$\bullet$] 
If $j<n$, then the only essential preimage of $uv$ under $f$ is an arc of the form $x_{i}x_{j-1}$. Hence the only essential preimage of $\gamma$ is an essential curve $\gamma'$   separating $x_{i}$ and $x_{j-1}$ from the other postcritical points, with $i<j-1$. By the inductive hypothesis, $\gamma'$ cannot be a member of a Levy cycle. Thus $\gamma$ cannot be a member of a Levy cycle either.

\item[$\bullet$] If $j=n$, then the only essential preimage of $uv$ under $f$ is an arc of the form $x_{i}w$, where $w$ is either $x_{n}$ or $x_{n-1}$. So the only essential preimage of $\gamma$ is an essential curve $\gamma'$  that separates $x_{i}$ and $x_{n-1}$, or $x_{i}$ and $x_{n-1}$, from the other postcritical points. By the inductive hypothesis, $\gamma'$ cannot be a member of a Levy cycle. Thus $\gamma$ cannot be a member of a Levy cycle either.
\end{itemize}
\end{proof}

\begin{prop}\label{prop:proB} Let $Q\subset\{x_{1},x_{2},\dots,x_{n}\}$ with $|Q|=k\geq 2$. If $\gamma$ is a simple closed curve that separates the sets $Q$ and $P_{f}\setminus Q$, then $\gamma$ cannot be a member of a Levy cycle.
\end{prop}
\begin{proof} We proceed by induction on $k=|Q|$.

\emph{Base case:} $k=2$. This is Proposition \ref{prop:proA}.

\emph{Inductive Hypothesis $(\star)$:} Suppose that the original statement is valid for any subset $Q$ of $\{x_{1},x_{2},\dots x_{n}\}$ with $|Q|=k-1\geq1$.

Now let $Q\subset\{x_{1},x_{2},\dots,x_{n}\}$ with $|Q|=k$. We may assume that $Q=\{x_{i_1},x_{i_{2}},\cdots, x_{i_{k}}\}$ with $i_{1}<i_{2}<\cdots<i_{k}\leq n$. Let $\gamma$ be a simple closed curve separating the sets $Q$ and $P_{f}\setminus Q$. Denote by $D_{\gamma}$ the connected component of $S^{2}\setminus\gamma$ that contains $Q$. We proceed by induction on $i_{1}$ to show that $\gamma$ cannot be member of a Levy cycle.

\begin{itemize}
\item[$\bullet$] \emph{Base case:} $i_{1}=1$. Due to the action $f$ on $P_{f}\setminus\{a,b\}$, every connected component of $f^{-1}(D_{\gamma})$ has at most $k$ postcritical points. We now analyze two subcases.

\begin{itemize}
\item Every essential connected component of $f^{-1}(D_{\gamma})$ has less than $k$ postcritical points. The connected component that contains $b$ separates $a$ and $b$, so the boundary of this connected component cannot be a member of a Levy cycle. On the other hand, the inductive hypothesis $(\star)$ implies the boundary of the essential components of $f^{-1}(D_\gamma)$ -if any- cannot be member of a Levy cycle. So $\gamma$ cannot be a member of a Levy cycle.
\item At least one connected component of $f^{-1}(D_{\gamma})$ contains $k$ postcritical points. In this subcase, exactly one connected component of $f^{-1}(D_{\gamma})$ is essential and this connected component contains the set $Q'=\{b,x_{i_{2}-1},x_{i_{3}-1},\cdots\}$. The boundary of this component, say $\gamma'$, is the only essential preimage of $\gamma$ and it separates the points $a$ and $b$. Hence $\gamma'$ cannot be homotopic to a member of a Levy cycle. Thus $\gamma$ cannot be a member of a Levy cycle either.
\end{itemize}

\item[$\bullet$] \emph{Inductive Hypothesis $(\star\star)$:} Suppose the statement is valid for some $i_{1}$. 
\item[$\bullet$] \emph{Inductive Step:} Suppose that $\gamma$ separates $Q=\{x_{i_{1}+1},x_{s_{2}},\cdots, x_{s_{k}}\}$ and $P_{f}\setminus Q$ with $i_{1}+1<s_{2}<\cdots<s_{k}\leq n$. We now show that $\gamma$ cannot be member of a Levy cycle. There are two possible scenarios.
\begin{itemize}

\item If every essential connected component of $f^{-1}(D_{\gamma})$ has less than $k$ postcritical points, then the inductive hypothesis $(\star)$ implies that none of the essential components of $f^{-1}(\gamma)$ is a member of a Levy cycle. So $\gamma$ cannot be a member of a Levy cycle. 
\item If at least one connected component of $f^{-1}(D_{\gamma})$ contains $k$ postcritical points, then exactly one connected component of $f^{-1}(D_{\gamma})$, say $D'$, is essential. If $s_{k}<n$, then $D'$ contains the set $Q'=\{x_{i_{1}},x_{s_{2}-1},\cdots, x_{s_{k}-1}\}$. If $s_{k}=n$, $D'$ contains either the set $\{x_{i_{1}},x_{s_{2}-1},\cdots, x_{n-1}\}$ or the set $\{x_{i_{1}},x_{s_{2}-1},\cdots, x_{n}\}$. The boundary of $D'$, say $\gamma'$ is the only essential preimage of $\gamma$. By the inductive hypothesis $(\star\star)$, $\gamma'$ cannot be a member of a Levy cycle. So $\gamma$ cannot be a member of a Levy cycle either. This completes the proof of the proposition.
\end{itemize}
\end{itemize}
\end{proof}

\begin{cor}\label{cor:corC} Let $f$ be a bicritical Thurston map that realizes portrait (4) in Theorem \ref{main:thm:1}, then $f$ does not have Levy cycles.

\end{cor}
\begin{proof} We proceed by contradiction. Suppose that there is a bicritical Thurston map $f$ realizing portrait (4) that admits a Levy cycle. By Lemma \ref{lemma:lem2}, the members of the Levy cycle must be essential curves that cannot separate the two critical values of $f$, say $a$ and $b$. So each member of the Levy cycle must separate a subset $Q$ of $\{x_{1},\cdots,x_{n}\}$ and $P_{f}\setminus Q$. This contradicts Proposition \ref{prop:proB}.
\end{proof}

\section{Portraits Realized by An Obstructed Map}\label{sect:por}
In this section, we prove Theorem \ref{main:thm:2}. 
 
To prove Theorem  \ref{main:thm:2}, we need a topological
description for bicritical maps. The following definitions are taken from \cite{FKKPS}. 

We define a \emph{rose} to be the boundary of the union of finitely many
closed topological disks in the 2-sphere which are disjoint except for
having exactly one point in common.  We view a rose as a graph with exactly one vertex. Its edges are called \emph{petals}. We define a rose map to be a map of pairs $g\co (S_1^2,R_1)\to
(S_2^2,R_2)$, where $S_1^2$ and $S_2^2$ are two oriented copies of
$S^2$, $g\co S_1^2\to S_2^2$ is an orientation-preserving finite
branched covering map, $R_2\subseteq S_2^2$ is a rose, $R_1=
g^{-1}(R_2)$ is a graph with pullback graph structure and every
connected component of $S_2^2\setminus R_2$ contains at most one
critical value of $g$. 

In the next lemma, $\Gamma$ is a bicritical portrait of degree $d\geq 2$, $P=P_{\Gamma}$, $\tau:V(\Gamma)\to V(\Gamma)$ is as in Section \ref{sect:preliminaries}, and $V_{\tau}$ is the set of critical values of $\tau$. Under the assumptions of Lemma \ref{lem:reductionlem} one can precompose a rose map $g\co S_1^2\to S_2^2$ that realizes the branch data of $\Gamma$ with a suitable homeomorphism $h\co S_2^2\to S_1^2$ to obtain a Thurston map $f:=g\circ h$ that realizes $\Gamma$ and has a Levy cycle.

\begin{lemma}\label{lem:reductionlem} Let $A\subset P$ such that $V_{\tau}\subseteq A$ and $|P\setminus A|\geq2$. Let $B=P\setminus A$. Suppose that there exists a partition $\{D_{1},\dots,D_{s}\}$ of the set $\tau^{-1}(B)$ with $s\leq d$ such that the restrictions $\tau|_{D_{i}}:D_{i}\to B$ are injections for $i=1,\dots,s$, and that $A\subseteq D_{1}$.
\begin{itemize}
\item[(1)] If $D_{1}=A$, then $\Gamma$ is realized by a Thurston map that has a Levy curve.
\item[(2)] If $D_{1}=A\sqcup\{c\}$ with $\tau(c)=c$, then $\Gamma$ is realized by a Thurston map that has a Levy cycle of length 2.
\end{itemize}
\end{lemma}

\begin{proof}  
Let $V_{\tau}=\{ v_{1},v_{2}\}$. Enumerate the vertices of $\Gamma$ as $v_{1},v_{2}, \cdots, v_{n}$ with $A=\{v_{1},\dots, v_{r}\}$. Since the branch data of $\Gamma$ is realized by the polynomial $z\mapsto z^{d}$, there exists a rose map $g:(S^2_{1},R_{1})\to(S^{2}_{2},R_{2})$ realizing the branch data of $\Gamma$ such that $V_{g}=\{v_1,v_2\}$,  the $n-1$ connected components of $S_{2}^{2}\setminus R_{2}$, each bounded by a petal of $R_{2}$, are labeled $v_{1},v_{3},\dots,v_{n}$ in counterclockwise order, and the remaining connected component is labeled as $v_{2}$. 
Let $\gamma$ be a simple closed curve in $S_{2}^{2}$ separating the sets $A$ and $B$. Let $D_{\gamma}$ be the open disk that contains the set $B$. Since $\overline{D_{\gamma}}\cap V_{g}=\emptyset$, $g^{-1}(D_{\gamma})$ is the disjoint union of $d$ disks, say $W_{1}$, $\dots$, $W_{d}$ and each restriction $g:{\overline{W_{i}}}\to\overline{D_{\gamma}}$ is a homeomorphism. Denote the restriction of $g$ to $W_{1}$ as $g_{1}$ and the boundary of $W_{1}$ as $\tilde{\gamma}$.

We now prove statement 1. Since $A=D_{1}$ and $\tau|_{D_{1}}:D_{1}\to B$ is an injection, $\tau(a)\in B$ for all $a\in A$ and the restriction $g_{1}^{-1}\circ\tau|_{A}: A\to W_{1}\cap g^{-1}(B)$ is an injection. Construct an orientation preserving homeomorphism $h: S^{2}_{2}\to S^2_{1}$ such that (i) $h:S^{2}_{2}\setminus D_{\gamma}\to\overline{W_{1}}$, (ii) $h(a)=g_{1}^{-1}(\tau(a))$ for all $a\in A$, and (iii) the dynamic portrait of $f:=g\circ h$ is isomorphic to $\Gamma$. The curve $\gamma$ is by itself a Levy cycle for $f$. This proves statement 1.

To prove statement 2, let $\alpha$ be a simple closed curve contained in $D_{\gamma}$ such that $\alpha$ separates the sets $A\sqcup\{c\}$ and $B\setminus\{c\}$ in $S^{2}_{2}$. Let $D_\alpha$ be the open disk that contains the set $B\setminus\{c\}$. Since $\overline{D_{\alpha}}\cap V_{g}=\emptyset$, $g^{-1}(D_{\alpha})$ is the disjoint union of $d$ disks, say $U_{1}$, $\dots$, $U_{d}$, and each restriction $g:{\overline{U_{i}}}\to\overline{D_{\alpha}}$ is a homeomorphism. We may and do assume that $\overline{U_{i}}\subset W_{i}$ for all $i$. 
Let $\tilde{\alpha}$ be the boundary of $U_{1}$. Now construct an orientation preserving homeomorphism $h: S^{2}_{2}\to S^2_{1}$ such that 
\begin{itemize}
\item[$\bullet$] $h(D_{\alpha})=S_{1}^{2}\setminus\overline{W_{1}}$, $h(a)=g^{-1}_{1}(\tau(a))$ for all $a\in A$,
\item[$\bullet$] $h(\overline{D}_{\gamma}\setminus D_{\alpha})=\overline{W_{1}}\setminus U_{1}$, $h(c)=g_{1}^{-1}(c)$, $h(\alpha)=\tilde{\gamma}$ $h(\gamma)=\tilde{\alpha}$,
\item[$\bullet$] $h(S^{2}_{2}\setminus\overline{D_{\gamma}})=U_{1}$, and $h(x)\in g^{-1}(\tau(x))$ for all $x\in B\setminus\{c\}$.
\end{itemize}
The dynamic portrait of $f:=g\circ h$ is isomorphic to $\Gamma$ and $\Delta=\{\gamma,\alpha\}$ is a Levy cycle for $f$. This proves statement 2. 
\end{proof}  

\begin{proof}[Proof of Theorem \ref{main:thm:2}]. We first prove the theorem in case (i). By Remark \ref{remark:rem1}, $\Gamma$ is realizable by a bicritical Thurston map, say $g$, but this Thurston map may not be obstructed. We may and do assume that $\Gamma=\Gamma_{g}$. Let $x_{1},x_{2},\dots,x_{n}$ be the points in the nonattracting cycle of length at least 2 of the dynamic portrait of $g$; i.e.,
\[\xymatrix{x_{1}\ar[r]&x_{2}\ar[r]&\cdots\ar[r]&x_{n}\ar@/^1pc/[lll]}\]

Because $g$ is bicritical, no point in $\{x_{1},\dots, x_{n}\}$ is a critical value. Fix a simple closed curve $\gamma$ that separates the sets $\{x_{1},\dots, x_{n}\}$ and $P_{g}\setminus\{x_{1},\dots, x_{n}\}$. Let $D_{\gamma}$ be the connected component of $S^2\setminus\gamma$ that contains the set $\{x_{1},\dots, x_{n}\}$. Then $g^{-1}(D_{\gamma})$ is the disjoint union of $d$ topological disks. Let $\widetilde{D}$ be the connected component of  $g^{-1}(D_{\gamma})$ that contains the point $x_{n}$. Let $\widetilde{\gamma}$ be boundary of $\widetilde{D}$. Since $g$ is open and proper, $\widetilde{\gamma}$ is a connected component of $g^{-1}(\gamma)$. For each $i\in\{1,\dots,n\}$, let $y_{i}$ be the unique element of $g^{-1}(x_{i})$ in $\widetilde{D}$. Note that $y_{1}=x_{n}$. Construct an orientation preserving homeomorphism $h:S^{2}\to S^{2}$ such that $h(x_{i})=y_{i+1\mod n}$ for $1\leq i\leq n$, $h(x)=x$ for all $x\in P_{g}\setminus\{x_{1},\dots, x_{n}\}$, and $h(D_{\gamma})=D_{\widetilde{\gamma}}$. Then $f:=h\circ g$ is a Thurston map whose dynamic portrait is isomorphic to $\Gamma_{g}$. Note that $\Gamma_{f}$ contains the 
 nonattracting cycle of length at least 2 given by
\[\xymatrix{y_{1}\ar[r]&y_{2}\ar[r]&\cdots\ar[r]&y_{n}\ar@/^1pc/[lll]}\]

By construction $\widetilde{\gamma}$ is by itself a Levy cycle for $f$.
 
In this paragraph we prove the theorem in case (ii) when $d\geq4$. Let $V_{\tau}=\{a,b\}$ be the set of critical values of $\Gamma$, let $P$ be the postcritical vertices of $\Gamma$ with $|P|\geq 4$. Let $A=V_{\tau}$ and $B=P\setminus A$. Since $d\geq4$ and $|\tau^{-1}(x)\cap P|\leq3$ for all $x\in P$, there exists a partition of the set $\tau^{-1}(B)$, say $\{D_{1},\dots, D_{s}\}$ with $s\leq d$ such that for each $i$ the restriction $\tau|_{D_{i}}: D_{i}\to B$ is an injection with $D_{1}=A$. Part(1) of Lemma  \ref{lem:reductionlem} implies that $\Gamma$ is realized by an obstructed Thurston map.

We now prove the theorem in case (ii) when $d=3$. There are two subcases:

\begin{itemize}
\item If $|\tau^{-1}(x)\cap P|<2$ for all $x\in P$, proceed as in case (ii) with $d\geq4$.
\item There is exactly one postcritical vertex whose three preimages are also postrcritical vertices. If such a postcritical vertex is not a fixed point, then the portrait has a nonattracting cycle. By case (i), the portrait admits an obstructed representative.  If such a postcritical vertex is a fixed point, then the portrait has the form
$$\xymatrix@R=1pc{
\star\ar[r]^{3}&a\ar[r]&x_{1}\ar[r] &x_{2}\ar[r]&\cdots\ar[r]&x_{n}\ar@(ur,dr)\\
\star\ar[r]^{3}&b\ar[r]&y_{1}\ar[r]&y_{2}\ar[r]&\cdots\ar[r]&y_{m}\ar[u] }$$
with $n\geq m\geq1$. Consider $A=V_{\tau}=\{a,b\}$ and let $B=P\setminus A$. If $n=1$, then $|P|=4$ and there exists a partition $\{D_{1}, D_{2}, D_{3}\}$ of the set $g^{-1}(B)$ such that $D_{1}=A$, $x_{1}\in D_{2}$, $y_{1}\in D_{3}$.  
 If $n\geq2$, then $|P|\geq5$ and there exists a partition $\{D_{1}, D_{2}, D_{3}\}$ of the set $g^{-1}(B)$ such that $x_{n}\in D_{1}$, $x_{n-1}\in D_{2}$, $y_{m}\in D_{3}$ with $D_{1}=A\sqcup\{x_{n}\}$. In the former subcase apply part(1) of Lemma \ref{lem:reductionlem} and in the latter subcase apply part(2) of Lemma \ref{lem:reductionlem} with $c=x_{n}$ to conclude that the portrait is realized by an obstructed Thurston map.\end{itemize}
\end{proof}
\section{Degree 2}\label{sect:deg2}
In this section, we determine the dynamic portraits of degree 2 that can be realized by an obstructed Thurston map. We first summarize some known partial results. 
\begin{itemize}
\item Any dynamic portrait of degree 2 with two connected components is realized by a Thurston map that is the formal mating of two Thurston polynomials $P(z)=z^{2}+c_{P}$ and $Q(z)=z^2+c_{Q}$, with $c_{P}$ and $c_{Q}$ in the 1/2-limb of the Mandelbrot set (see Section \ref{sect:preliminaries}). By Theorem 4.1 of \cite{TanL}, such a formal mating has a Levy cycle. 
\item By Theorem 1.10 in \cite{FKKPS}, the only bicritical polynomial portraits of degree $d\geq2$ that are completely unobstructed are portraits (1) and (2) in Theorem \ref{main:thm:1}.
\item By Theorem\ref{main:thm:2}, if $\Gamma$ has a nonattracting cycle of length at least 2, then $\Gamma$ is realized by a Thurston map that has a Levy cycle.
 \end{itemize}

In the rest of this section we assume that $\Gamma$ is a connected dynamic portrait of degree $2$, it is not a  polynomial portrait and has no nonattracting cycles of length at least $2$. So, either $\Gamma$ is connected and has a fixed point or $\Gamma$ is connected and has an attracting cycle. In each of the following cases $P=P_{\Gamma}$, $A$ is a specific subset of the vertices of $\Gamma$, $B=P\setminus A$ and $\tau^{-1}(B)=D_{1}\sqcup D_{2}$. Applying 
Lemma \ref{lem:reductionlem} we show that if $\Gamma$ is not in the list provided in Theorem \ref{main:thm:1}, then $\Gamma$ can be realized by an obstructed Thurston map.

\noindent{\bf Case 1. $\Gamma$ is connected and has exactly one fixed point.} 

\begin{itemize}
\item[I.] Both critical points of $\Gamma$ are preperiodic.
$$\xymatrix@R=0.05pc{
\star \ar[r]^2 & x_{1} \ar[r] &\cdots \ar[r] & x_{m} \ar[dr] & & & & \\
& & & & z_1 \ar[r] & \cdots \ar[r] & z_\ell \ar@(ur,dr) [] \\
\star \ar[r]^2 & y_{1} \ar[r] &\cdots\ar[r] & y_{n} \ar[ur] & & & &
}$$
Here $\ell\geq2$ and by symmetry we may assume that $m\geq n$. If $(m,n)=(1,1)$, then $\Gamma$ is  isomorphic to portrait (6) in Theorem \ref{main:thm:1} which is completely unobstructed. So we assume that $(m,n)\neq(1,1)$.

\begin{itemize}
\item If $n\geq2$ and $\ell\neq3$, consider $A=\{x_1,y_1,z_1,z_{\ell-1}\}$ and $D_{1}=\{x_{1}, y_{1}, z_{1},z_{\ell-1}\}$. By part(1) of Lemma \ref{lem:reductionlem}, $\Gamma$ is realized by an obstructed map.

\item If $n\geq2$ and $\ell=3$, consider $A=\{x_1,y_1,z_1\}$ and $D_{1}=\{x_{1},y_{1},z_{1},z_{3}\}$. By part(2) of Lemma \ref{lem:reductionlem} with $c=z_{3}$, $\Gamma$ is realized by an obstructed map.

\item If $n=1$ and $\ell\geq2$, consider $A=\{x_{1},y_{1}\}$ and $D_{1}=\{x_{1},y_{1},z_{\ell}\}$. By part(2) of Lemma \ref{lem:reductionlem} with $c=z_{\ell}$, $\Gamma$ is realized by an obstructed map.
\end{itemize}

\item[II.] Exactly one critical point of $\Gamma$ is preperiodic.
$$\xymatrix{\star\ar[r]^{2}&x_1\ar[r]&\cdots\ar[r]&x_{k}\ar[r]^{2}&x_{k+1}\ar[r]&\cdots\ar[r]&x_{n-1}\ar[r]&x_{n}\ar@(ur,dr)}$$
If $k=1$, then $\Gamma$ is isomorphic to  portrait (4) in Theorem \ref{main:thm:1} which is completely unobstructed. So we assume that $k>1$ and $|P|=n\geq4$. Note that $k+2\leq n$.
\begin{itemize}
\item If $n\geq4$ and $k=n-2$, consider $A=\{x_{1},x_{k+1}\}$ and $D_{1}=A$. By part(1) of Lemma \ref{lem:reductionlem}, $\Gamma$ is realized by an obstructed map.

\item If $n\geq5$ and $2\leq k\leq n-3$, take $A=\{x_{1},x_{k+1}\}$ and $D_{1}=\{x_{1},x_{k+1}, x_{n}\}$. By part(2) of Lemma \ref{lem:reductionlem} with $c=x_{n}$, $\Gamma$ is realized by an obstructed map.
\end{itemize}
\end{itemize}

\noindent{\bf Case 2. $\Gamma$ is connected and has one attracting cycle.} 
\begin{itemize}
\item[I.] {The attracting cycle contains both critical points.} 
$$\xymatrix{x_{n}\ar[r]^{2}&x_{1}\ar[r]&\cdots\ar[r]\ar[r]&x_{k}\ar[r]^{2}&x_{k+1}\ar[r]&\cdots\ar[r]&x_{n-1}\ar@/^1pc/[llllll]}$$\\
If $k\in\{1,n-1\}$, then $\Gamma$ is isomorphic to portrait (3) in Theorem \ref{main:thm:1} which is completely unobstructed. So we assume that $n\geq 4$ and $1<k<n-1$. Under these settings, consider $A=\{x_1,x_{k+1}\}$ and $D_{1}=A$. Then apply part(1) of Lemma \ref{lem:reductionlem} to conclude that $\Gamma$ is realized by an obstructed map.

\item[II.]{ A critical point is preperiodic and no critical point has two preimages in $P$.}
$$\xymatrix{\star\ar[r]^{2}&x_{1}\cdots\ar[r]&x_{m}\ar[r]&\mu\ar[r]&y_{1}\cdots\ar[r]\ar[r]&y_{k}\ar[r]^{2}&y_{k+1}\cdots\ar[r]&y_{n}\ar@/^1pc/[llll]}$$\\
If $(m,n)=(1,k+1)$, then $\Gamma$ is isomorphic to portrait (5) in Theorem\ref{main:thm:1} which is completely unobstructed. So we assume that $(m,n)\neq(1,k+1)$. 
\begin{itemize}
\item If $m=1$ and $k+1<n$, consider $A=\{x_1,y_{k+1}\}$ and $D_{1}=A$. 
\item If $m\geq2$ and $k+1=n$, consider $A=\{x_1,y_{k+1}\}$ and $D_{1}=A$. 
\item If $m\geq2$ and $k+1<n$, consider $A=\{x_1,\mu,y_{k+1}\}$ and $D_{1}=A$.
\end{itemize}
In any scenario, apply part(1) of Lemma \ref{lem:reductionlem} to conclude that $\Gamma$ is realized by an obstructed map.

\item[III.]{ A critical point is preperiodic and the other critical point has two preimages in $P$.}
$$\xymatrix{\star\ar[r]^{2}&x_{1}\ar[r]&\cdots\ar[r]&x_{m}\ar[r]&\mu\ar[r]^{2}&y_{1}\ar[r]&\cdots\ar[r]&y_{n}\ar@/^1pc/[lll]}$$\\
If $(m,n)=(2,2)$, then $\Gamma$ is isomorphic to portrait (7) in Theorem \ref{main:thm:1} which is completely unobstructed. So we assume that $(m,n)\neq(2,2)$.
\begin{itemize}
\item If ($m=1$ and $n\geq2$) or $(m,n)=(2,1)$, consider $A=\{x_1,y_{1}\}$ and $D_{1}=A$. 
\item If $m\geq3$ and $n=2$, consider $A=\{x_{1},y_{1},x_{m}\}$ and $D_{1}=A$.
\item If $m\geq2$ and $n\geq3$, consider $A=\{x_{1},y_{1},y_{n}\}$ and $D_{1}=A$.
\end{itemize}
In any scenario, apply part(1) of Lemma \ref{lem:reductionlem} to conclude that $\Gamma$ is realized by an obstructed map.
\end{itemize}

\section{Completely Unobstructed Bicritical Portraits}\label{sect:cubp}
In this section, we prove that the only completely unobstructed bicritical dynamic portraits are the ones listed in Theorem \ref{main:thm:1}. More precisely, 

\begin{cor}\label{cor:corD}  Let $\Gamma$ be an abstract bicritical portrait of degree $d\geq2$ with at least four postcritical points. Then $\Gamma$ is completely unobstructed if and only if $\Gamma$ is isomorphic to one of the dynamic portraits listed in Theorem \ref{main:thm:1}.
\end{cor} 
\begin{proof} Theorem \ref{main:thm:1} shows that portraits (1)-(7) are completely unobstructed and the discussion in Section\ref{sect:deg2} shows that these are precisely the completely unobstructed dynamic portraits in degree 2. Now, let $\Gamma$ be an abstract bicritical portrait of degree $d\geq3$ with at least four postcritical points and assume that $\Gamma$ is completely unobstructed. We will show that $\Gamma$ is isomorphic to one of the dynamic portraits (1)-(6) listed in Theorem \ref{main:thm:1}.

By the classification of abstract polynomial portraits (see Theorem 1.10 in \cite{FKKPS}), the only completely unobstructed polynomial portraits are precisely portraits (1) and (2) listed in Theorem \ref{main:thm:1}. From now on we assume that $\Gamma$ is a bicritical portrait without fixed critical points. Let $a$ and $b$ be the critical values of $\Gamma$ and denote by $\tau(a)=p$ and $\tau(b)=q$. If $\deg(a)=\deg(b)=1$ and $p\neq q$, case (ii) in Theorem \ref{main:thm:2} implies that $\Gamma$ can be realized by an obstructed Thurston map. Hence, either $\deg(a)=\deg(b)=1$ and $\tau(a)=\tau(b)=p$, or one of the critical values, say $a$, is a critical point and $\tau(a)=b$.
\begin{itemize}
\item[Case 1.] $\deg(a)=\deg(b)=1$ and $\tau(a)=\tau(b)=p$. In this case, $\Gamma$ has the form

\begin{center}
\begin{tikzcd}[column sep=2em, row sep=0.01pc]
\star \arrow[r, "d"] & a \arrow[dr] & & & & \\
& & p \arrow[r] &\cdots \\
\star \arrow[r, "d"] & b \arrow[ur] & & & &
\end{tikzcd}
\end{center}
Due to case (i) in Theorem \ref{main:thm:2}, the forward orbit of $p$ cannot contain a nonattracting cycle of length at least 2. So, either the forward orbit of $p$ contains a fixed point or  contains exactly one critical point. The former subcase implies that $\Gamma$ is isomorphic to portrait (6) and the latter subcase implies that $\Gamma$ is isomorphic to portrait (5).\vspace{.1in}
\item[Case 2.] One of the critical values, say $a$, is a critical point and $\tau(a)=b$.
$$\xymatrix{x\ar[r]^{d}&a\ar[r]^{d}&b\ar[r]&q\ar[r]&\cdots}$$
Due to case (i) in Theorem \ref{main:thm:2}, the forward orbit of $q$ cannot contain a nonattracting cycle of length at least 2. So, either the forward orbit of $q$ contains a fixed point or contains the critical point labeled by $x$. The former subcase implies that $\Gamma$ is isomorphic to portrait (4) and the latter subcase implies that $\Gamma$ is isomorphic to portrait (3). 
\end{itemize}
This completes the proof of the corollary.
\end{proof}

\noindent {\bf Acknowledgment.} The first author would like to thank Walter Parry for multiple insightful discussions and for introducing him to the problem of realizability of bicritical portraits.

\end{document}